\documentclass{birkjour}

\usepackage[cp1251]{inputenc}
\usepackage[english]{babel}
\usepackage{latexsym,amsfonts,amssymb,amsmath,epsfig,mathrsfs}
\usepackage{graphicx,graphics,hhline}
\usepackage{euscript}
\usepackage[hidelinks]{hyperref}

\usepackage{tikz-cd}
\usetikzlibrary{shapes,arrows}
\tikzstyle{block} = [rectangle, draw,
text width=7em, text centered, rounded corners,draw=black, minimum height=2.5em,text centered]
\tikzstyle{line} = [draw, -latex', thick]
\tikzset{node distance = 2cm}

\newtheorem{theorem}{Theorem}
\newtheorem{proposition}[theorem]{Proposition}

\begin{document}

\title[The convergence of sequences]
	{The convergence of sequences in terms of positive and alternating Perron expansions}

\author{M. Moroz 
}
\address{Institute of Mathematics\\ 
	National Academy of Sciences of Ukraine\\
	Tereschenkivska 3\\
	Kyiv\\
	Ukraine}
\email{moroznik22@gmail.com}

\date{}
\dedicatory{}
\thanks{}

\begin{abstract}
We consider conditions for the convergence of sequences in terms of positive and alternating Perron expansions ($P$-representation and $P^-$-representation). These conditions are crucial to determine the continuity of functions that are defined using $P$-representation or $P^-$-representation of real numbers.
\end{abstract}


\maketitle


\section{Introduction}

In \cite{Moroz2024}, the author proposed the representation of real numbers by positive Perron series \cite{Perron} ($P$-representation). This representation generalizes L\"{u}roth expansion \cite{ZhPr2012}, Engel expansion \cite{ERS1958}, Sylvester expansion \cite{ERS1958}, Dar\'oczy--K\'atai-Birthday expansion \cite{Galambos1998}, and others, and it is also related to Oppenheim expansions \cite{Galambos1976,Oppenheim1972,SydorukTorbin2017}. In \cite{Moroz2024PP}, the author proposed the representation of real numbers by alternating Perron series ($P^-$-representation), which generalizes Pierce expansion \cite{BPT2013,Shallit1986}, second Ostrogradsky expansion \cite{TorbinPratsyovyta2010}, alternating L\"{u}roth expansion \cite{KalpazidouKnopfmacher1990}, etc. Furthermore, in \cite{Moroz2024PP}, it is proved that when calculating the Lebesgue measure of sets defined using Perron expansions (both positive and alternating), it is sufficient to consider only positive Perron expansions. In other words, Lebesgue measure problems defined in terms of the $P^-$-representation can be reduced to similar problems defined using the $P$-representation, and are therefore not of particular scientific interest.

However, topological properties of the $P^-$-representation are fundamen\-tally different from properties of the $P$-representation. Therefore, the $P^-$-re\-presentation is as interesting as the $P$-representation from the perspective of theory of functions. Both representations can be effectively used to construct and investigate functions with locally complicated structures and fractal properties, such as nowhere monotonic or nowhere differentiable functions, singular functions, and others that at the same time are everywhere or almost everywhere continuous functions, transformations preserving measure or fractal dimension, etc. This is confirmed by the successful use of partial cases of the $P$-representation and the $P^-$-representation for such purposes (for example, see \cite{ABKP2013,BP2023,BPT2013,Moroz2017,Moroz2024}).

In this paper, we establish conditions for the convergence of sequences in terms of the $P$-representation and the $P^-$-representation of numbers. These conditions are important to determine the continuity of functions defined in terms of positive and alternating Perron expansions. Some of these conditions have already implicitly arisen during the investigation of certain functions \cite{ABKP2013,BP2023,BPT2013,Moroz2017}, but they have not yet been generalized to the case of Perron expansions or fully systematized. We hope that this work will fill this gap and prove useful in the investigation of fractal functions.

\section{General information about Perron expansions}

Let us consider functions $\varphi_n(x_1,\ldots,x_n)\colon\mathbb{N}^ n\rightarrow\mathbb{N}$ for all $n\in\mathbb{N}$, $\varphi_0\equiv\text{const}\in\mathbb{N}$. A fixed sequence of functions $(\varphi_n)_{n=0}^\infty$ is denoted by $P$. Each such sequence $P$ defines two types of Perron expansions of real numbers: the $P$-representation (positive Perron expansion) and the $P^-$-representation (alterna\-ting Perron expansion). 

\textbf{The $P$-representation of real numbers.} The $P$-representation of $x\in(0,1]$ is defined as the expansion of $x$ into a positive Perron series:
\begin{equation}\label{1}
	x=\sum_{n=0}^{\infty}\frac{r_0 \cdots r_n}{(p_1-1)p_1\cdots(p_n-1)p_n p_{n+1}},
\end{equation}
where $r_0=\varphi_0$, $r_n=\varphi_n(p_1,\ldots,p_n)$, and  $p_{n}\geq r_{n-1}+1$ for all $n\in\mathbb{N}$. 
The number $x$ and its $P$-representation \eqref{1} are briefly denoted by $\Delta^P_{p_1 p_2 \ldots}$. For each sequence $P$ of functions $\varphi_n$, every $x\in(0,1]$ has a unique $P$-representation {\cite[Theorem 1]{Moroz2024}}. The number $p_n=p_n(x)$ is called the $n$th $P$-digit of $x$.

The $P$-cylinder of rank $k$ with base $c_1 \ldots c_k$ is defined as the set $\Delta^P_{c_1 \ldots c_k}$ of all numbers $x\in(0,1]$ such that $p_1(x)=c_1,\ldots,p_k(x)=c_k$. It is known  {\cite[Lemma 3, Corollary 3]{Moroz2024}} that the $P$-cylinder $\Delta^P_{c_1 \ldots c_k}$ is a set of the type $(a,b]$. The infimum, supremum, and diameter of $P$-cylinder $\Delta^P_{c_1 \ldots c_k}$ are calculated using the following formulas:
\begin{gather}
\inf \Delta^P_{c_1 \ldots c_k}=\sum_{n=0}^{k-1}\frac{r_0\cdots r_{n}}{(c_1-1)c_1\cdots(c_{n}-1)c_{n}c_{n+1}},\label{infPerron2}\\
\sup \Delta^P_{c_1 \ldots c_k}=\inf \Delta^P_{c_1 \ldots c_k}+\frac{r_0 \cdots r_{k-1}}{(c_1-1)c_1\cdots(c_k-1)c_k}, \label{supPerron2}\\		
|\Delta^P_{c_1 \ldots c_k}|=\frac{r_0 \cdots r_{k-1}}{(c_1-1)c_1\cdots(c_k-1)c_k},\label{cylPerron2}
\end{gather}
where $r_0=\varphi_0$ and $r_n=\varphi_n(c_1,\ldots,c_n)$ for all $n=1,\ldots,k-1$. 

In more detail, the geometry of the $P$-representation is presented in \cite{Moroz2024}. Figures \ref{fig:1} and \ref{fig:2} illustrate the relative positions of $P$-cylinders, providing an idea of the geometry of the $P$-representation.

\begin{figure}[h]
\begin{center}
	\begin{tikzpicture}[out=145,in=35]
		\draw[{Circle[scale=.7]}-] (8.5,0) -- (5.55,0)
		node[pos=0.01,below] {1} node[pos=1,below] {$\frac{r_0}{r_0+1}$};
		\draw[shorten >=1.2pt][-{Straight Barb[width'=1.5,scale=2.5]}] (8.5,0) to node[above] {$\Delta^{P}_{r_0+1}$} (5.5,0);
		\draw[{Circle[scale=.7]}-] (5.55,0) -- (3.05,0)
		node[pos=1,below] {$\frac{r_0}{r_0+2}$};
		\draw[shorten >=1.2pt][-{Straight Barb[width'=1.5,scale=2.5]}] (5.5,0) to node[above] {$\Delta^{P}_{r_0+2}$} (3,0);
		\draw[{Circle[scale=.7]}-] (3.05,0) -- (1.55,0)
		node[pos=1,below] {$\frac{r_0}{r_0+3}$};
		\draw[shorten >=1.2pt][-{Straight Barb[width'=1.5,scale=2.5]}] (3,0) to node[above] {$\Delta^{P}_{r_0+3}$} (1.5,0);
		\draw[{Circle[scale=.7]}-{Circle[open,scale=.7]}] (1.55,0) -- node[above=1] {$\ldots$} node[below=2.5] {$\cdots$} (0,0)
		node[pos=0.97,below] {0};
	\end{tikzpicture}
\end{center}
	\caption{\small $P$-cylinders of rank $1$ inside $(0,1]$, $r_0=\varphi_0$.}
	\label{fig:1}
\end{figure}
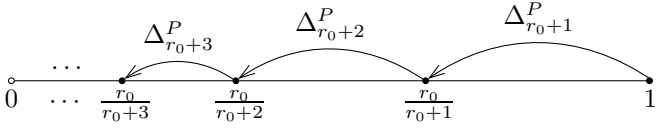

\begin{figure}[h]
\begin{center}
	\begin{tikzpicture}[out=145,in=35]
		\draw[{Circle[scale=.7]}-] (8.5,0) -- (5.5,0)
		node[pos=0,below] {$\sup\Delta^{P}_{c_1\ldots c_k}$};
		\draw[shorten >=1.2pt][-{Straight Barb[width'=1.5,scale=2.5]}] (8.5,0) to node[above] {$\Delta^{P}_{c_1\ldots c_k [r_k+1]}$} (5.5,0);
		\draw[{Circle[scale=.7]}-] (5.55,0) -- (3.05,0);
		\draw[shorten >=1.2pt][-{Straight Barb[width'=1.5,scale=2.5]}] (5.5,0) to node[above] {$\Delta^{P}_{c_1\ldots c_k [r_k+2]}$} (3,0);
		\draw[{Circle[scale=.7]}-] (3.05,0) -- (1.55,0);
		\draw[shorten >=1.2pt][-{Straight Barb[width'=1.5,scale=2.5]}] (3,0) to node[above] {$\Delta^{P}_{c_1\ldots c_k [r_k+3]}$} (1.5,0);
		\draw[{Circle[scale=.7]}-{Circle[open,scale=.7]}] (1.55,0) -- node[above=1] {$\ldots$} (0,0)
		node[pos=1,below] {$\inf\Delta^{P}_{c_1\ldots c_k}$};
		\draw[out=-170,in=-10,shorten >=2pt] (8.5,0) to node[below] {$\Delta^{P}_{c_1\ldots c_k}$} (0,0);
	\end{tikzpicture}
\end{center}
	\caption{\small $P$-cylinders of rank $k+1$ inside the $P$-cylinder $\Delta^{P}_{c_1\ldots c_k}$ of rank $k$, $r_k=\varphi_k(c_1,\ldots,c_k)$.}
\label{fig:2}
\end{figure}
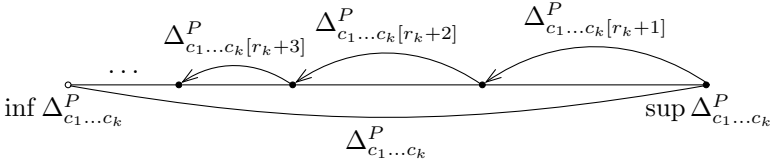

\textbf{The $P^-$-representation of real numbers.} The $P^-$-representation of $x\in(0,1]$ is defined as expansion of $x$ into an alternating Perron series:
\begin{equation}\label{2}
	x=\sum_{n=0}^{\infty}\frac{(-1)^n r_0 \cdots r_n}{(q_1-1)p_1\cdots(q_n-1)q_n (q_{n+1}-1)},
\end{equation}
where $r_0=\varphi_0$, $r_n=\varphi_n(q_1,\ldots,q_n)$, and $q_{n}\geq r_{n-1}+1$ for all $n\in\mathbb{N}$. The number $x$ and its $P^-$-representation \eqref{2} are briefly denoted by $\Delta^{P^-}_{q_1 q_2 \ldots}$.

For each sequence $P$ of functions $\varphi_n$, there exists a countably infinite set $IS^{P^-}\subset(0,1]$ whose elements do not have its $P^-$-representation. However, every $x\in(0,1)\setminus IS^{P^-}$ has a unique $P^-$-representation {\cite[Theorem 3.12]{Moroz2024PP}}. For $x\in(0,1)\setminus IS^{P^-}$, the number $q_n=q_n(x)$ is called the $n$th $P^-$-digit of $x$.

The $P^-$-cylinder of rank $k$ with base $c_1 \ldots c_k$ is defined as the set $\Delta^{P^-}_{c_1 \ldots c_k}$ of all numbers  $x\in(0,1)\setminus IS^{P^-}$ such that $q_1(x)=c_1,\ldots,q_k(x)=c_k$. It is known {\cite[Corollary 3.14]{Moroz2024PP}} that the $P^-$-cylinder $\Delta^{P^-}_{c_1 \ldots c_k}$ is a set of the type $(a,b)\setminus IS^{P^-}$. The infimum, supremum, and diameter of $P^-$-cylinder $\Delta^{P^-}_{c_1 \ldots c_k}$ are calculated using the following formulas:
\begin{itemize}
	\item if $k$ is odd, then
	\begin{gather}
		\sup \Delta^{P^-}_{c_1\ldots c_k}=\sum_{n=0}^{k-1}\frac{(-1)^n r_0 \cdots r_{n}}{(c_1-1)c_1\cdots (c_n-1)c_n (c_{n+1}-1)}, \label{sup1alternatingPerron2}\\
		\inf \Delta^{P^-}_{c_1 \ldots c_k}=\sup \Delta^{P^-}_{c_1 \cdots c_k}-\frac{r_0\cdots r_{k-1}}{(c_1-1)c_1\cdots (c_{k}-1)c_k},\label{inf1alternatingPerron2}
	\end{gather}
	\item if $k$ is even, then
	\begin{gather}
		\inf \Delta^{P^-}_{c_1 \ldots c_k}=\sum_{n=0}^{k-1}\frac{(-1)^n r_0 \cdots r_{n}}{(c_1-1)c_1\cdots (c_n-1)c_n (c_{n+1}-1)}, \label{inf2alternatingPerron2}\\		
		\sup \Delta^{P^-}_{c_1 \ldots c_k}=\inf \Delta^{P^-}_{c_1 \ldots c_k}+\frac{r_0\cdots r_{k-1}}{(c_1-1)c_1\cdots (c_{k}-1)c_k},\label{sup2alternatingPerron2}
	\end{gather}
	\item in both cases
	\begin{gather}
		|\Delta^{P^-}_{c_1 \ldots c_k}|=\frac{r_0\cdots r_{k-1}}{(c_1-1)c_1\cdots (c_{k}-1)c_k},\label{cylalternatingPerron2}	
	\end{gather}
\end{itemize}
where $r_0=\varphi_0$ and $r_n=\varphi_n(c_1,\ldots,c_n)$ for all $n=1,\ldots,k-1$. It is important that $IS^{P^-}$ is the set of infima and suprema of all $P^-$-cylinders.

In more detail, the geometry of the $P^-$-representation is presented in \cite{Moroz2024PP}. Figures \ref{fig:3}--\ref{fig:5} illustrate the relative positions of $P^-$-cylinders, providing an idea of the geometry of the $P^-$-representation. Note that the geometry of the $P^-$-representation is similar to that of continued fractions.

\begin{figure}[h]
\begin{center}
	\begin{tikzpicture}[out=145,in=35]
		\draw[dashed, {Circle[open,scale=.7]}-] (8.5,0) -- (5.55,0)
		node[pos=0.01,below] {1} node[pos=1,below] {$\frac{r_0}{r_0+1}$};
		\draw[shorten <=1.2pt,shorten >=1.2pt] (8.47,0) to node[above] {$\Delta^{P^-}_{r_0+1}$} (5.5,0);
		\draw[dashed, {Circle[open,scale=.7]}-] (5.55,0) -- (3.05,0)
		node[pos=1,below] {$\frac{r_0}{r_0+2}$};
		\draw[shorten <=1.2pt,shorten >=1.2pt] (5.5,0) to node[above] {$\Delta^{P^-}_{r_0+2}$} (3,0);
		\draw[dashed, {Circle[open,scale=.7]}-] (3.05,0) -- (1.55,0)
		node[pos=1,below] {$\frac{r_0}{r_0+3}$};
		\draw[shorten <=1.2pt,shorten >=1.2pt] (3,0) to node[above] {$\Delta^{P^-}_{r_0+3}$} (1.5,0);
		\draw[dashed, {Circle[open,scale=.7]}-{Circle[open,scale=.7]}] (1.55,0) -- node[above=1] {$\ldots$} node[below=2.5] {$\cdots$} (0,0)
		node[pos=0.97,below] {0};
	\end{tikzpicture}
\end{center}
	\caption{\small $P^-$-cylinders of rank $1$ inside $(0,1]$, $r_0=\varphi_0$.}
\label{fig:3}
\end{figure}
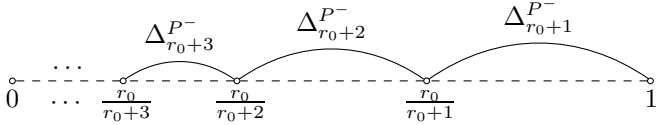

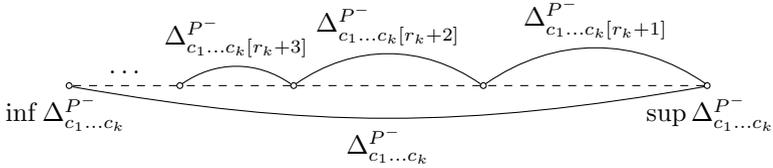
\begin{figure}[h]
\begin{center}
	\begin{tikzpicture}[out=145,in=35]
		\draw[dashed, {Circle[open,scale=.7]}-] (8.5,0) -- (5.55,0)
		node[pos=0,below] {$\sup\Delta^{P^-}_{c_1\ldots c_k}$};
		\draw[shorten <=1.2pt,shorten >=1.2pt] (8.47,0) to node[above] {$\Delta^{P^-}_{c_1\ldots c_k [r_k+1]}$} (5.5,0);
		\draw[dashed, {Circle[open,scale=.7]}-] (5.55,0) -- (3.05,0);
		\draw[shorten <=1.2pt,shorten >=1.2pt] (5.5,0) to node[above] {$\Delta^{P^-}_{c_1\ldots c_k [r_k+2]}$} (3,0);
		\draw[dashed, {Circle[open,scale=.7]}-] (3.05,0) -- (1.55,0);
		\draw[shorten <=1.2pt,shorten >=1.2pt] (3,0) to node[above] {$\Delta^{P^-}_{c_1\ldots c_k [r_k+3]}$} (1.5,0);
		\draw[dashed, {Circle[open,scale=.7]}-{Circle[open,scale=.7]}] (1.55,0) -- node[above=1] {$\ldots$} (0,0)
		node[pos=1,below] {$\inf\Delta^{P^-}_{c_1\ldots c_k}$};
		\draw[shorten <=2.5pt,shorten >=2.5pt][out=-10,in=-170] (0,0) to node[below] {$\Delta^{P^-}_{c_1\ldots c_k}$} (8.5,0);
	\end{tikzpicture}
\end{center}
	\caption{\small $P^-$-cylinders of \textbf{odd} rank $k+1$ inside the $P^-$-cylinder $\Delta^{P^-}_{c_1\ldots c_k}$ of \textbf{even} rank $k$, $r_k=\varphi_k(c_1,\ldots,c_k)$.}
\label{fig:4}
\end{figure}

\begin{figure}[h]
\begin{center}
	\begin{tikzpicture}[out=35,in=145]
		\draw[dashed, {Circle[open,scale=.7]}-] (0,0) -- (2.95,0)
		node[pos=0,below] {$\inf\Delta^{P^-}_{c_1\ldots c_k}$};
		\draw[shorten <=1.2pt,shorten >=1.2pt] (0.03,0) to node[above] {$\Delta^{P^-}_{c_1\ldots c_k [r_k+1]}$} (3,0);
		\draw[dashed, {Circle[open,scale=.7]}-] (2.95,0) -- (5.45,0);
		\draw[shorten <=1.2pt,shorten >=1.2pt] (3,0) to node[above] {$\Delta^{P^-}_{c_1\ldots c_k [r_k+2]}$} (5.5,0);
		\draw[dashed, {Circle[open,scale=.7]}-] (5.45,0) -- (6.95,0);
		\draw[shorten <=1.2pt,shorten >=1.2pt] (5.5,0) to node[above] {$\Delta^{P^-}_{c_1\ldots c_k [r_k+3]}$} (7,0);
		\draw[dashed, {Circle[open,scale=.7]}-{Circle[open,scale=.7]}] (6.95,0) -- node[above=1] {$\ldots$} (8.5,0)
		node[pos=1,below] {$\sup\Delta^{P^-}_{c_1\ldots c_k}$};
		\draw[shorten <=2.5pt,shorten >=2.5pt][out=-10,in=-170] (0,0) to node[below] {$\Delta^{P^-}_{c_1\ldots c_k}$} (8.5,0);
	\end{tikzpicture}
\end{center}
	\caption{\small $P^-$-cylinders of \textbf{even} rank $k+1$ inside \\ the $P^-$-cylinder $\Delta^{P^-}_{c_1\ldots c_k}$ of \textbf{odd} rank $k$, $r_k=\varphi_k(c_1,\ldots,c_k)$.}
\label{fig:5}
\end{figure}
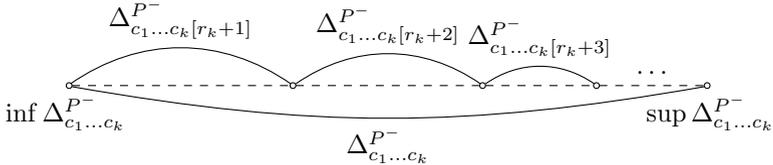

\section{The convergence of sequences in terms of $P$-representation}

Let us prove several propositions that establish a connection between the convergence of a numerical sequence and the $P$-digits of its elements.

Let $x_0=\Delta^P_{c_1 c_2 \ldots}\in(0,1]$, let $(x_n)_{n=1}^\infty$ be a sequence of numbers from $(0,1]\setminus\{x_0\}$, and let $k_n$ be the minimum number such that $p_{k_n}(x_n)\neq p_{k_n}(x_0)$. 

\begin{proposition}\label{lemmalim1P}
	If $\displaystyle\lim_{n\to\infty}k_n=\infty$, then $\displaystyle\lim_{n\to\infty}x_n=x_0$.
\end{proposition}

\begin{proof}
Assume, without loss of generality, that $k_n\geq 2$. Then $x_0$ and $x_n$ are contained in the same $P$-cylinder $\Delta^P_{c_1\ldots c_{k_n-1}}$ of rank $k_n-1$, and we have
	$$0<\left|x_0-x_n\right|<\left|\Delta^P_{c_1\ldots c_{k_n-1}}\right|.$$
Since $k_n\to\infty$, and given {\cite[Lemma 4]{Moroz2024}}, we have $\left|\Delta^P_{c_1\ldots c_{k_n-1}}\right|\to 0$ as $n\to\infty$. It follows that $x_n\to x_0$.
\end{proof}

The condition $\displaystyle\lim_{n\to\infty}k_n=\infty$ is sufficient for the convergence of the sequence $(x_n)_{n=1}^\infty$ to $x_0$. However, it is generally not necessary. There are several basic cases, each with its own necessary and sufficient conditions for convergence.

\begin{proposition}\label{lemmalim2P}
If $x_0$ is not the endpoint of some $P$-cylinder, then 	$$\displaystyle\lim_{n\to\infty}x_n=x_0\iff\lim_{n\to\infty}k_n=\infty.$$
\end{proposition}

\begin{proof}
	Given the Proposition \ref{lemmalim1P}, it suffices to prove that
	 $$\lim_{n\to\infty}x_n=x_0\implies\lim_{n\to\infty}k_n=\infty.$$
Assume the opposite; that is, $\displaystyle\lim_{n\to\infty}x_n=x_0$ and $\displaystyle\lim_{n\to\infty}k_n\neq\infty$. Then there exists a number $M$ and an infinite sequence $(n_t)_{t=1}^\infty$ such that $k_{n_t}<M$. Since $x_0$ is neither the infimum nor the supremum of any $P$-cylinder, we have 
	$$\inf\Delta^P_{c_1 \ldots c_M}<x_0=\Delta^P_{c_1 c_2 \ldots }<\sup\Delta^P_{c_1 \ldots c_M}.$$
It is clear that $x_{n_t}\notin\Delta^P_{c_1 \ldots c_{M}}$, and therefore
	$$\left|x_0-x_{n_t}\right|\geq \min\left\{x_0-\inf\Delta^P_{c_1 \ldots c_M},\sup\Delta^P_{c_1 \ldots c_M}-x_0 \right\}=\text{const}>0.$$
	Since the last inequality is true for an infinite number of indices $n_t$, then $x_n\nrightarrow x_0$. We have reached a contradiction. Therefore, $k_n\to\infty$.
\end{proof}

It follows from the geometry of the $P$-representation that the infimum of any $P$-cylinder is the supremum of the neighboring left $P$-cylinder of the same rank. 
Also, the supremum of any $P$-cylinder, except for  $\Delta^P_{[r_0+1]}$, is the infimum of another $P$-cylinder, although not necessarily of the same rank. 

The following proposition provides the necessary and sufficient conditions for the sequence $(x_n)_{n=1}^\infty$ to converge to $x_0$ from the left, where $x_0$ is the supremum of a $P$-cylinder.

\begin{proposition}
	Let $x_0\in(0,1]$ be a supremum of a $P$-cylinder, and let the sequence $(x_n)_{n=1}^\infty$ be such that $0<x_n<x_0$ for all $n\in\mathbb{N}$. Then $$\displaystyle\lim_{n\to\infty}x_n=x_0\iff\lim_{n\to\infty}k_n=\infty.$$
\end{proposition}

\begin{proof}
The proof is similar to that of Proposition \ref{lemmalim2P}.
\end{proof}

The following proposition provides the necessary and sufficient conditions for the sequence $(x_n)_{n=1}^\infty$ to converge to $x_0$ from the right, where $x_0$ is the infimum of a $P$-cylinder.

\begin{proposition}
Let $x_0=\inf\Delta^P_{c_1\ldots c_k}$, and let the sequence $(x_n)_{n=1}^\infty$ be such that $x_0<x_n\leq 1$ for all $n\in\mathbb{N}$. Then   $\displaystyle\lim_{n\to\infty}x_n=x_0$ if and only if there exists a number $n_0$ such that
	\begin{gather}\label{systemconvergence}
		\begin{cases}
			p_i(x_n)=c_i \text{ for all } i\leq k \text{ and } n\geq n_0,\\
			\displaystyle\lim_{n\to\infty}p_{k+1}(x_n)=\infty.
		\end{cases}
	\end{gather}
\end{proposition}

\begin{proof}
\emph{Sufficiency.} It follows from \eqref{systemconvergence} that  $$x_n\in\Delta^P_{c_1\ldots c_k p_{k+1}(x_n)}\subset\Delta^P_{c_1\ldots c_k}$$
for all $n\geq n_0$. Then $\inf\Delta^P_{c_1\ldots c_k}<x_n\leq\sup\Delta^P_{c_1\ldots c_k p_{k+1}(x_n)}$. Using {\cite[Corollary 1]{Moroz2024}}, which is one of the consequences of \eqref{infPerron2} and \eqref{supPerron2}, we obtain
	\begin{align*}
		0&<x_n-x_0\leq\sup\Delta^P_{c_1\ldots c_k p_{k+1}(x_n)}-\inf\Delta^P_{c_1\ldots c_k}\\
		&=\frac{r_0 r_1\cdots r_k}{(c_1-1)c_1\cdots(c_k-1)c_k(p_{k+1}(x_n)-1)}\to 0
	\end{align*}
as $p_{k+1}(x_n)\to \infty$, where $r_0=\varphi_0$ and $r_i=\varphi_i(c_1,\ldots,c_i)$ for all $i\leq k$. Then $x_n\to x_0$.

\emph{Necessity.} Assume that \eqref{systemconvergence} fails to hold for any $n_0$. We consider two possible cases and show that in each case $x_n\nrightarrow x_0$.

\emph{Case 1: there exists a sequence $(n_t)_{t=1}^\infty$ such that $p_{i_t}(x_{n_t})\neq c_{i_t}$ for some $i_t\leq k$ for all $t\in\mathbb{N}$.} In this case, $x_{n_t}\not\in\Delta^P_{c_1\ldots c_k}$. Then
	\begin{gather*}
		x_0=\inf\Delta^P_{c_1\ldots c_k}<\sup\Delta^P_{c_1\ldots c_k}<x_{n_t},\\
		x_{n_t}-x_0>\left|\Delta^P_{c_1\ldots c_k}\right|=\text{const}>0.
	\end{gather*}
Since the last inequality is true for an infinite number of indices $n_t$, then $x_n\nrightarrow x_0$.

\emph{Case 2: $p_i(x_n)=c_i$ for all $i\leq k$ and  $n\geq n_0$, but there exist a number $M$ and a sequence $(n_t)_{t=1}^\infty$ such that $p_{k+1}(x_{n_t})< M$.} Assume, without loss of generality, that $n_t\geq n_0$. Then, using {\cite[Corollary 1]{Moroz2024}}, we obtain:
	\begin{align*}
		x_{n_t}-x_0&\geq\inf\Delta^P_{c_1\ldots c_k p_{k+1}(x_{n_t})}-\inf\Delta^P_{c_1\ldots c_k}\\
		&=\frac{r_0 r_1\cdots r_k}{(c_1-1)c_1\cdots(c_k-1)c_k p_{k+1}(x_{n_t})}\\
		&>\frac{r_0 r_1\cdots r_k}{(c_1-1)c_1\cdots(c_k-1)c_k M}=\text{const}>0,
	\end{align*}
where $r_0=\varphi_0$ and $r_i=\varphi_i(c_1,\ldots,c_i)$ for all $i\leq k$. Since the last inequality is true for an infinite number of indices $n_t$, then $x_n\nrightarrow x_0$.

Therefore, the condition \eqref{systemconvergence} is necessary and sufficient for the conver\-gence of the sequence $(x_n)_{n=1}^\infty$ to $x_0$.
\end{proof}

Since the number $0$ does not have its $P$-representation, we will consider the convergence of the sequence $(x_n)_{n=1}^\infty$ to $0$ separately.

\begin{proposition}\label{lemmalim0P}
	$\displaystyle\lim_{n\to\infty}x_n=0\iff\lim_{n\to\infty}p_1(x_n)=\infty.$
\end{proposition}

\begin{proof}
This proposition follows from the inequality $0<\frac{r_0}{p_1(x_n)}<x_n\leq\frac{r_0}{p_1(x_n)-1},$ where $r_0=\varphi_0$.
\end{proof}


\section{The convergence of sequences in terms of $P^-$-representation}

Not every $x\in(0,1)$ has its $P^-$-representation. Therefore, we will assume that all numbers in the sequence $(x_n)_{n=1}^\infty$ belong to $(0,1)\setminus IS^{P^-}$. However, the limit of such a sequence can be any number from $[0,1]$. Because of this, we need to consider several cases.

Let $x_0=\Delta^{P^-}_{c_1 c_2 \ldots}\in(0,1)\setminus IS^{P^-}$, let $(x_n)_{n=1}^\infty$ be a sequence of numbers from $(0,1)\setminus\left(IS^{P^-}\cup\{x_0\}\right)$, and let $k_n$ be the minimum number such that $q_{k_n}(x_n)\neq q_{k_n}(x_0)$.

\begin{proposition}\label{lemmalim1Palt}
	$\displaystyle\lim_{n\to\infty}x_n=x_0\iff\lim_{n\to\infty}k_n=\infty.$
\end{proposition}

\begin{proof}
	The proof is similar to that of Proposition \ref{lemmalim1P} and Proposition \ref{lemmalim2P}.
\end{proof}

Consider the cases when the sequence $(x_n)_{n=1}^\infty$ converges to a number that does not have its $P^-$-representation.

\begin{proposition}
	$\displaystyle\lim_{n\to\infty}x_n=0\iff\lim_{n\to\infty}q_1(x_n)=\infty$.
\end{proposition}

\begin{proof}
	The proof is similar to that of Proposition \ref{lemmalim0P}.
\end{proof}

It follows from the geometry of the $P^-$-representation that every number from $IS^{P^-}$ is the supremum for some $P^-$-cylinder of odd rank. In particular, it follows from \eqref{sup1alternatingPerron2}--\eqref{sup2alternatingPerron2} that
\begin{itemize}
\item $\inf\Delta^{P^-}_{c_1 \ldots c_k}=\sup\Delta^{P^-}_{c_1 \ldots c_{k-1} [c_k+1]}$ if $k$ is odd,
\item $\sup\Delta^{P^-}_{c_1 \ldots  c_k}=\sup\Delta^{P^-}_{c_1 \ldots c_{k} [r_k+1]}$ if $k$ is even, $r_k=\varphi_k(c_1,\ldots,c_k)$,
\item $\inf\Delta^{P^-}_{c_1 \ldots  c_k}=\inf\Delta^{P^-}_{c_1 \ldots  c_{k-1}}=\sup\Delta^{P^-}_{c_1 \ldots  [c_{k-1}+1]}$ if $k$ is even and $c_k=\varphi_{k-1}(c_1,\ldots,c_{k-1})+1$,
\item $\inf\Delta^{P^-}_{c_1 \ldots  c_k}=\sup\Delta^{P^-}_{c_1 \ldots  c_{k-1}[c_k-1]}=\sup\Delta^{P^-}_{c_1 \ldots c_{k-1}[c_k-1] [r_k+1]}$ if $k$ is even and $c_k>\varphi_{k-1}(c_1,\ldots,c_{k-1})+1$, $r_k=\varphi_k(c_1,\ldots,c_{k-1},c_k-1)$.
\end{itemize}

\begin{proposition}\label{lemmalim2Palt}
	Let $x_0=\sup \Delta^{P^-}_{c_1\ldots c_k}$, where $k$ is odd, and let the sequence $(x_n)_{n=1}^\infty$ be such that $0<x_n<x_0$ for all $n\in\mathbb{N}$. Then $\displaystyle\lim_{n\to\infty}x_n=x_0$ if and only if there exists a number $n_0$ such that
	\begin{gather}\label{systemconvergencealt1}
		\begin{cases}
			q_i(x_n)=c_i \text{ for all } i\leq k \text{ and } n\geq n_0,\\
			\displaystyle\lim_{n\to\infty}q_{k+1}(x_n)=\infty.
		\end{cases}
	\end{gather}
\end{proposition}

\begin{proof}
	\emph{Sufficiency.} It follows from \eqref{systemconvergencealt1} that $$x_n\in\Delta^{P^-}_{c_1\ldots c_k q_{k+1}(x_n)}\subset\Delta^{P^-}_{c_1\ldots c_k}$$
	for all $n\geq n_0$. Then $\inf\Delta^{P^-}_{c_1\ldots c_k q_{k+1}(x_n)}<x_n<\sup\Delta^{P^-}_{c_1\ldots c_k}$. Using {\cite[Corollary 3.8]{Moroz2024PP}}, which is one of the consequences of the formulas \eqref{sup1alternatingPerron2}--\eqref{sup2alternatingPerron2}, we obtain 
	\begin{align*}
		0&<x_0-x_n<\sup\Delta^{P^-}_{c_1\ldots c_k}-\inf\Delta^{P^-}_{c_1\ldots c_k q_{k+1}(x_n)}\\
		&=\frac{r_0 r_1\cdots r_k}{(c_1-1)c_1\cdots(c_k-1)c_k(q_{k+1}(x_n)-1)}\to 0
	\end{align*}
	as $q_{k+1}(x_n)\to \infty$, where $r_0=\varphi_0$ and $r_i=\varphi_i(c_1,\ldots,c_i)$ for all $i\leq k$. Then $x_n\to x_0$.
	
	\emph{Necessity.} Assume that \eqref{systemconvergencealt1} fails to hold for any $n_0$. We consider two possible cases and show that in each case $x_n\nrightarrow x_0$. 
	
	\emph{Case 1: there exists a sequence $(n_t)_{t=1}^\infty$ such that $q_{i_t}(x_{n_t})\neq c_{i_t}$ for some $i_t\leq k$ for all $t\in\mathbb{N}$.} In this case, $x_{n_t}\not\in\Delta^{P^-}_{c_1\ldots c_k}$. Then
	\begin{gather*}
		x_{n_t}<\inf\Delta^{P^-}_{c_1\ldots c_k}<\sup\Delta^{P^-}_{c_1\ldots c_k}=x_0,\\
		x_0-x_{n_t}>\left|\Delta^{P^-}_{c_1\ldots c_k}\right|=\text{const}>0.
	\end{gather*}
	Since the last inequality is true for an infinite number of indices $n_t$, then $x_n\nrightarrow x_0$.
	
	\emph{Case 2: $q_i(x_n)=c_i$ for all $i\leq k$ and $n\geq n_0$, but there exist a number $M$ and a sequence $(n_t)_{t=1}^\infty$ such that $q_{k+1}(x_{n_t})< M$.} Assume, without loss of generality, that $n_t\geq n_0$. Then, using {\cite[Corollary 3.8]{Moroz2024PP}}, we obtain:
	\begin{align*}
		x_0-x_{n_t}&>\sup\Delta^{P^-}_{c_1\ldots c_k}-\sup\Delta^{P^-}_{c_1\ldots c_k q_{k+1}(x_{n_t})}\\
		&=\frac{r_0 \cdots r_k}{(c_1-1)c_1\cdots(c_k-1)c_k q_{k+1}(x_{n_t})}\\
		&>\frac{r_0\cdots r_k}{(c_1-1)c_1\cdots(c_k-1)c_kM}=\text{const}>0,
	\end{align*}
	where $r_0=\varphi_0$ and $r_i=\varphi_i(c_1,\ldots,c_i)$ for all $i\leq k$. Since the last inequality is true for an infinite number of indices $n_t$, then $x_n\nrightarrow x_0$.
	
	Therefore, the condition \eqref{systemconvergencealt1} is necessary and sufficient for the conver\-gence of the sequence $(x_n)_{n=1}^\infty$ to $x_0$.
\end{proof}

It follows from the geometry of the $P^-$-representation that every number from $IS^{P^-}$, except for the number $1$, is the infimum for some $P^-$-cylinder of even rank. In particular, it follows from \eqref{sup1alternatingPerron2}--\eqref{sup2alternatingPerron2} that
\begin{itemize}
	\item $\sup\Delta^{P^-}_{c_1 \ldots c_k}=\inf\Delta^{P^-}_{c_1 \ldots c_{k-1} [c_k+1]}$ if $k$ is even,
	\item $\inf\Delta^{P^-}_{c_1 \ldots  c_k}=\inf\Delta^{P^-}_{c_1 \ldots c_{k} [r_k+1]}$ if $k$ is odd, $r_k=\varphi_k(c_1,\ldots,c_k)$,
	\item $\sup\Delta^{P^-}_{c_1 \ldots  c_k}=\sup\Delta^{P^-}_{c_1 \ldots  c_{k-1}}=\inf\Delta^{P^-}_{c_1 \ldots  [c_{k-1}+1]}$ if $k$ is odd and $c_k=\varphi_{k-1}(c_1,\ldots,c_{k-1})+1$,
	\item $\sup\Delta^{P^-}_{c_1 \ldots  c_k}=\inf\Delta^{P^-}_{c_1 \ldots  c_{k-1}[c_k-1]}=\inf\Delta^{P^-}_{c_1 \ldots c_{k-1}[c_k-1] [r_k+1]}$ if $k$ is odd and $c_k>\varphi_{k-1}(c_1,\ldots,c_{k-1})+1$, $r_k=\varphi_k(c_1,\ldots,c_{k-1},c_k-1)$.
\end{itemize}

\begin{proposition}
	Let $x_0=\inf \Delta^{P^-}_{c_1\ldots c_k}$, where $k$ is even, and let the sequence $(x_n)_{n=1}^\infty$ be such that $x_0<x_n<1$ for all $n\in\mathbb{N}$. Then $\displaystyle\lim_{n\to\infty}x_n=x_0$ if and only if there exists a number $n_0$ such that
	\begin{gather}\label{systemconvergencealt3}
		\begin{cases}
			q_i(x_n)=c_i \text{ for all } i\leq k \text{ and } n\geq n_0,\\
			\displaystyle\lim_{n\to\infty}q_{k+1}(x_n)=\infty.
		\end{cases}
	\end{gather}
\end{proposition}

\begin{proof}
	The proof is similar to that of Proposition \ref{lemmalim2Palt}.
\end{proof}

The above propositions are sufficient to determine the convergence of any sequence defined in terms of positive or alternating Perron expansions. If a sequence contains infinitely many elements on both sides of the limit, which is the endpoint of some $P$-cylinder or is a member of $IS^{P^-}$, it is sufficient to consider two corresponding subsequences and apply the relevant proposition to each of them. Also, to determine the continuity of a function at some point, the propositions about one-sided limits are quite sufficient.

\subsection*{Acknowledgements}
This work was supported by a grant from the Simons Foundation (1290607, M.M.).


\end{document}